\documentclass[reqno,11pt]{amsart}
\usepackage{geometry}
\usepackage[numbers,sort&compress]{natbib}
\usepackage{mathtools,amssymb,amsthm,mathrsfs,color,lineno,paralist,graphicx,float}
\usepackage[colorlinks,
linkcolor=red,
anchorcolor=green,
citecolor=blue,
]{hyperref}

\usepackage[T1]{fontenc}
\usepackage[utf8]{inputenc}

\usepackage{tikz}
\usetikzlibrary{positioning}
\usepackage{calc}
\definecolor{bleu1}{RGB}{0,57,128}
\def\bleu1{\color{bleu1}}
\usepackage{etoolbox}
\patchcmd{\section}{\normalfont}{\normalfont \bleu1}{}{}
\patchcmd{\subsection}{\normalfont}{\normalfont \bleu1}{}{}
\patchcmd{\subsubsection}{\normalfont}{\normalfont \bleu1}{}{}
\renewcommand{\proofname}{\it \bleu1 Proof}

\def\e{\varepsilon}

\let\newpf\proof \let\proof\relax 
\newenvironment{pf}{\newpf[\proofname]}{\qed\endtrivlist}

\newcommand{\ba}{\overline{A}}

\def\be{\begin{equation}}
\def\ee{\end{equation}}

\def\ba{{\begin{align}}}
\def\ea{{\end{align}}}

\def\bm{\begin{matrix}}
\def\em{\end{matrix}}

\def\0{{\mathbf 0}}

\newtheorem{Theorem}{Theorem}[section]
\newtheorem{Lemma}{Lemma}[section]
\newtheorem{Proposition}{Proposition}[section]

\numberwithin{equation}{section}

\theoremstyle{definition}

\def\tr{{\text{tr}}}

\newcommand{\Q}{{\mathbb Q}}
\newcommand{\R}{{\mathbb R}}
\newcommand{\T}{{\mathbb T}}

\newcommand{\Z}{{\mathbb Z}}

\def\B0{{\bold{0}}}

\catcode`\@=12

\def\Empty{}
\newcommand\oplabel[1]{
  \def\OpArg{#1} \ifx \OpArg\Empty {} \else
    \label{#1}
  \fi}

\newcommand{\comm}[1]{}
\newcommand{\comment}[1]{}

\begin{document}
\title[]{The Dry Ten Martini Problem for subcritical Type I operators}

\author {Lingrui Ge}
\address{Beijing International Center for Mathematical Research, Peking University, Beijing, China
	} \email{gelingrui@bicmr.pku.edu.cn}

      \begin{abstract}
We supply the details of the subcritical dry Ten Martini statement announced in \cite[Remark 1.3]{gjy}: for every irrational frequency, every subcritical Type I energy with resonant rotation number is an endpoint of an open spectral gap. The proof combines the all-frequency Puig argument developed there with quantitative almost reducibility. A different proof was recently given by Li \cite{L}.
\end{abstract}

\maketitle

\section{Introduction}
The subcritical dry Ten Martini statement was announced in
\cite{gjy}, where it was explained that it follows by
combining the all-frequency Puig argument developed there with
quantitative almost reducibility.   The
purpose of the present note is to supply the details of the argument
announced in \cite{gjy}. A different proof was recently
given by Li \cite{L}, based on a resolvent factorization of the
hyperbolic projection for the finite-range dual operators. 

Consider the following analytic one-frequency Schr\"odinger operator,
\begin{align}\label{sch}
(H_{v,\alpha,x}u)_n=u_{n+1}+u_{n-1}+ v(x+n\alpha)u_n,\ \ n\in\Z,
\end{align} 
where $\alpha\in\R\backslash\Q$, $x\in\T$ and $v\in C^\omega(\T,\R)$. The Lyapunov exponent of complexified  Schr\"odinger cocycles is defined as
\begin{align}\label{multiergodicsch}
L_\e(E)=\lim\limits_{n\rightarrow\infty}\frac{1}{n}\int_\T\ln \|A_E(x+i\e+(n-1)\alpha)\cdots A_E(x+i\e)\|dx
\end{align}
where
\begin{equation}\label{S}
A_E(x)=\begin{pmatrix}E-v(x)&-1\\ 1&0\end{pmatrix}.
\end{equation}
The  {\it acceleration} \cite{avila0} and   {\it T-acceleration} \cite{gjy}  are defined as 
$$
\omega(E)=\lim\limits_{\e\rightarrow
  0^+}\frac{L_\e(E)-L_0(E)}{2\pi\e},\ \ 
\bar{\omega}(E)=\lim\limits_{\e\rightarrow \e_1^+}\frac{L_\e(E)-L_{\e_1}(E)}{2\pi(\e-\e_1)}
$$
where $0\le\e_1<\infty$ is the first turning point of the piecewise affine
function $L_\e(E)$.  
\begin{Theorem}\label{8}
Assume $\alpha\in \R\backslash\Q$, $v\in C_{h_1}^\omega(\T,\R)$, $2\rho(\alpha,A_{E_{k_0}})-k_0\alpha\in\Z$, $L(E_{k_0})=\omega(E_{k_0})=0$ and $\bar{\omega}(E_{k_0})=1$, then $E_{k_0}$ is an endpoint of an open spectral gap carrying this label.
\end{Theorem}

\section{Key ingredients}

\subsection{Quantitative almost reducibility}
Let $0<h_*<h_s$ where $h_s$ is the subcritical strip of $(\alpha,A_E)$, $0<\alpha<1$ and $\delta'>0$ be the  constants from Theorem 3.2 in \cite{dryAYZ}. Let
\begin{equation}
\varepsilon_{n} = 2e^{\frac{2\pi h_{*}}{1+\alpha}}e^{-q_{n}\delta'/2},\qquad
\delta_{n} = h_{*}/2^{n},\qquad
h_{n} = h_{*}/(1+\alpha)-\sum_{k=1}^n\delta_{k}/6,
\end{equation}
and also $\mathfrak{L}=\begin{pmatrix}0&1\\0&0\end{pmatrix}.$

\begin{Theorem}[\cite{arc1,dryAYZ}]
\label{thm3.2}
Let $\alpha\in\mathbb{R}\backslash\mathbb{Q}$ with $\beta(\alpha)>0$, $k_{0}\in\mathbb{Z}$ and suppose that $(\alpha,A)$ is subcritical and $\ell=2\rho_{f}(\alpha,A)-k_{0}\alpha\in\mathbb{Z}$. Then there exist a subsequence of $q_{n}$ with $q_{n+1}>e^{(\beta-o(1))q_{n}}$ \footnote{If $\beta(\alpha)=\infty$, consider the sequence such that $q_{n+1}>e^{100q_n}$.}, $B_{n}(x)\in C_{h_{n}}^{\omega}(\mathbb{T},\mathrm{PSL}(2,\mathbb{R}))$ with $\deg B_{n}=k_{0}$ and $d_{n}\in\mathbb{R}$ such that
\begin{equation}\label{1}
(-1)^{\ell}B_{n}(x+\alpha)^{-1}A(x)B_{n}(x)=\mathrm{id}+d_{n}\mathfrak{L}+F_{n}(x)
\end{equation}
\begin{equation}\label{5}
|B_{n}|_{h_{n}}\leq e^{2q_{n+1}\varepsilon_{n}^{\frac{1}{4}}},\ \ |d_{n}|\leq e^{-q_{n+1}\varepsilon_{n}^{\frac{1}{4}}},\ \ |F_{n}|_{h_{n}}\leq e^{-q_{n+1}\delta_{n}}.
\end{equation}
\end{Theorem}

\begin{Theorem}[\cite{aj1,arc2}]
\label{thm3.3}
Let $\alpha\in\mathbb{R}\backslash\mathbb{Q}$ with $\beta(\alpha)=0$, $k_{0}\in\mathbb{Z}$ and suppose that $(\alpha,A)$ is subcritical and $\ell=2\rho_{f}(\alpha,A)-k_{0}\alpha\in\mathbb{Z}$. Then $(\alpha,A)$ is reducible, i.e. there exists $B(x)\in C^{\omega}(\mathbb{T},\mathrm{PSL}(2,\mathbb{R}))$, $d\in\mathbb{R}$ such that
\begin{equation}
(-1)^{\ell}B(x+\alpha)^{-1}A(x)B(x)=\mathrm{id}+d\mathfrak{L}.
\end{equation}
\end{Theorem}

\subsection{All-frequency Puig's argument}
The following theorem is essentially proved in \cite{gjy}, we give a sketch here. 
\begin{Theorem}\label{6}
The following statements hold:

{\rm (1)}If we assume $\beta(\alpha)=0$, $E_{k_0}$ is a subcritical type I energy and $2\rho(\alpha,A_{E_{k_0}})-k_0\alpha\in\Z$, applying Theorem \ref{thm3.3} to $A_{E_{k_0}}$, we further have $d\neq 0$.

{\rm (2)}If we assume $\beta(\alpha)>0$, $E_{k_0}$ is a subcritical type I energy and $2\rho(\alpha,A_{E_{k_0}})-k_0\alpha\in\Z$, applying Theorem \ref{thm3.2} to $A_{E_{k_0}}$, we further have $|d_n|\geq   e^{-\frac{1}{100}q_{n+1}\delta_{n}}$ for $n$ sufficiently large.
\end{Theorem}
\begin{pf}
$\beta(\alpha)=0$ case is proved in Theorem 11.2 in \cite{gjy}. For $\beta(\alpha)>0$, suppose that the claimed bound fails along infinitely many admissible indices. We will consider the case \((-1)^{\ell}=1\) and $B_k\in C^\omega(\T,SL(2,\R))$, the other case being analogous  \footnote{One may need to adjust the argument to $\theta=0,\frac{1}{2}, \frac{1}{2}+\frac{\alpha}{2}, \frac{\alpha}{2}$ for different situations.}. Assume that $|d_k|\leq  e^{-\frac{1}{100}q_{k+1}\delta_{k}}$. Set $h=\lim\limits_{k\rightarrow \infty}h_k$, $h_1>\e_1(E)$ the analytic radius of $v$, $0<\eta\ll 1$,  
$$
n_k=[\frac{\delta_k}{4000 h_1}q_{k+1}],\ \ r_k=[\frac{1000}{\eta}\e_k^{\frac{1}{4}}q_{k+1}], \ \ 
I_{k}=\left[-\frac{[r_k]}{h},\frac{[r_k]}{h}\right].
$$

We denote
\begin{equation}\label{2}
B_k(x)=\begin{pmatrix}
b^k_{11}(x)&b^k_{12}(x)\\
b^k_{21}(x)&b^k_{22}(x)
\end{pmatrix},\ \ \begin{pmatrix}e^k_{11}(x)&e^k_{12}(x)\\
e^k_{21}(x)&e^k_{22}(x)
\end{pmatrix}=B_k(x+\alpha)\left(d_{k}\mathfrak{L}+F_{k}(x)\right).
\end{equation}
By  \eqref{1}, for $j=1,2$ we have
\begin{equation}\label{subge}
-b^k_{1j}(x+\alpha)-b^k_{1j}(x-\alpha)+(E-v(x))b^k_{1j}(x)=e^k_{1j}(x)-e^k_{2j}(x-\alpha):=g_j^k(x).
\end{equation}
By \eqref{5} and \eqref{2}, we have
\begin{equation}\label{tan7}
|B_{k}|_{h}\leq Ce^{\eta r_k},\ \ |g^k_{j}|_h\leq C e^{-10n_kh_1}.
\end{equation}

Let  $b^k_{1j}(x)=\sum_m\hat{b}^k_j(m)e^{-2\pi i mx}$ be the Fourier
expansion. By \eqref{subge}, for $j=1,2$, 
\begin{equation}\label{cot3}
\left((L^{2\cos}_{v_{n_k},\alpha,0}-E)\hat{b}^k_j\right)(m)=-\hat{g}^k_j(m)-\sum_{|l|\geq n_k+1} \hat{v}_l \hat{b}_j^k(m+l)=:\hat{f}_j^k(m).
\end{equation}
By \eqref{tan7} and \eqref{cot3}, we have
\begin{align}\label{cot4}
\nonumber |\hat{f}_j^k(m)|&\leq |\hat{g}_j^k(m)|+C(\sum_{l=n_k+1}^\infty e^{-2\pi|l|h_1}|B^k|_0+\sum_{l=-\infty}^{-n_k-1}e^{-2\pi |l|h_1}|B^k|_0)\\
&\leq e^{-10h_1 n_k}+Ce^{r_k}e^{-2\pi(n_k+1)h_1}\leq Ce^{-2\pi n_k (h_1-\delta)}
\end{align}
for any fixed $\delta>0$ and large $k$.
\begin{Lemma}[\cite{dryAYZ,gjy}]\label{10}
For $j=1,2$, there exists $m_j\in I_k$, such that
\begin{equation}\label{z1-estimate-21age}
|\hat{b}^k_{j}(m_j)|\geq ce^{-\eta r_k}.
\end{equation}
\end{Lemma}

Compared with \cite{gjy}, the Fourier lower bound $r_k^{-1}$
is replaced by $ce^{-\eta r_k}$, while the uniform
bound on $|B^k|_0$ is replaced by $Ce^{\eta r_k}$.
These changes introduce factors $e^{O(\eta r_k)}$
in the proof of Proposition~11.1 and the subsequent
argument, which are absorbed by the exponential
margins when $\eta>0$ is sufficiently small. Then we have all the desired estimates in Theorem 11.2 in \cite{gjy}. In the final truncation argument, choose
$s_k=\min\{l(n_k), Mr_k\}$, where $M$ is a
sufficiently large fixed integer. The final contradiction follows from $r_k=o(n_k)$. Thus we get a contradiction.
\end{pf}
\subsection{Generalized Moser-P\"oschel's argument} The following proposition is essentially proved in \cite{dryAYZ}, we also give a sketch here.
\begin{Proposition}\label{9}
Suppose that \(\alpha\in\mathbb R\backslash\mathbb Q\), and the cocycle
\((\alpha,A_{E_{k_0}})\) satisfies Theorem \ref{6}, then the cocycle \((\alpha,A_{E_{k_0}+\tau})\) is uniformly hyperbolic with
\[
2\rho(\alpha,A_{(E_{k_0}+\tau)}(x))
=
k_0\alpha\bmod\mathbb Z
\]   
\begin{itemize}
\item for $0<|\tau|<\tau_0$, satisfying \(d\tau<0\) where $\tau_0>0$ is sufficiently small, if  $\beta(\alpha)=0$;
\item for $e^{-\frac{\delta_n}{5}q_{n+1}}<|\tau|\leq e^{-\frac{\delta_n}{10}q_{n+1}}$ satisfying \(d_n\tau<0\) for all sufficiently large admissible $n$ if $\beta(\alpha)>0$.
\end{itemize} 
\end{Proposition}

\begin{pf}
$\beta(\alpha)=0$ follows from the classical Moser-P\"oschel argument. For $\beta(\alpha)>0$,
we will consider the case \((-1)^{\ell}=1\), the other case being analogous. Applying Theorem \ref{thm3.2} to $(\alpha,A_{E_{k_0}})$, we get a sequence of $B_n$ and $F_n$, We write
\begin{equation}
B_n(x)=
\begin{pmatrix}
z_{11}(x) & z_{21}(x)\\
z_{12}(x) & z_{22}(x)
\end{pmatrix},
\qquad
F_n(x)=
\begin{pmatrix}
\beta_{1}(x) & \beta_{2}(x)\\
\beta_{3}(x) & \beta_{4}(x)
\end{pmatrix}.
\end{equation}
Then, following the proof of  Proposition 5.1 in \cite{dryAYZ}, one can conjugate the cocycle \((\alpha,A_{(E_{k_0}+\tau)}(x))\) to \((\alpha,A_2)\) by a degree $k_0$ map with
\[
A_2=
\begin{pmatrix}
1+\tau\bigl([z_{11}z_{21}]-d_n[z_{11}^{2}]\bigr)
&
d_n+\tau\bigl(-d_n[z_{11}z_{21}]+[z_{21}^{2}]\bigr)\\[2mm]
-\tau[z_{11}^{2}]
&
1-\tau[z_{11}z_{21}]
\end{pmatrix}
+\widetilde M,
\]
where
\[
\begin{aligned}
\|\widetilde M\|_{0}\leq |\tau|^{2}e^{8q_{n+1}\epsilon_n^{1/4}}+e^{-\delta_nq _{n+1}}\leq 
3|\tau|^{2}e^{8q_{n+1}\epsilon_n^{1/4}}, \ \  e^{-\frac{\delta_n}{5}q_{n+1}}<|\tau|\leq e^{-\frac{\delta_n}{10}q_{n+1}}.
\end{aligned}
\]

The \(x\)-independent part of the cocycle \((\alpha,A_2)\), i.e.
\((\alpha,A_2-\widetilde M)\), is hyperbolic if \(d_n\tau<0\). In fact,
\[
\tr (A_2-\widetilde M)-2\geq c^2|d_n||\tau| |B_n|_0^{-2} \geq c^2|\tau|e^{-\frac{\delta_n}{100}q_{n+1}}
e^{-4q_{n+1}\epsilon_n^{1/4}},
\]
\[
 |\det{(A_2-\widetilde M)}-1|\leq \tau^2e^{8q_{n+1}\e_n^{\frac{1}{4}}},\ \ \|\widetilde M\|_{0}\leq 3|\tau|^{2}e^{8q_{n+1}\epsilon_n^{1/4}}
\]
Thus it is easy to see that the cocycle \((\alpha,A_2(x))\) is uniformly hyperbolic when \(d_n\tau<0\) and $e^{-\frac{\delta_n}{5}q_{n+1}}<|\tau|\leq e^{-\frac{\delta_n}{10}q_{n+1}}$, by standard invariant cone argument. Hence \((\alpha,A_{(E_{k_0}+\tau)})\) is also uniformly hyperbolic  and
\[
2\rho(\alpha,A_{(E_{k_0}+\tau)}(x))
=
k_0\alpha\bmod\mathbb Z.
\]
\end{pf}
\noindent Proof of Theorem \ref{8}:  
Choose $\tau$ as in Proposition \ref{9} in either case. $E_{k_0}+\tau\notin\Sigma$, and, for
sufficiently small $|\tau|$, equality of the rotation numbers
modulo $\mathbb Z$ gives equality of their continuous local
lifts. Hence
\[
N(E_{k_0}+\tau)=N(E_{k_0}).
\]
By monotonicity, $N$ is constant between these two energies.
Since $\operatorname{supp}(dN)=\Sigma$, the open interval
between them is disjoint from $\Sigma$. Finally,
$L(E_{k_0})=0$ implies $E_{k_0}\in\Sigma$, so $E_{k_0}$ is an endpoint
of an open spectral gap carrying the prescribed label.\qed

\section{AI-use disclosure}
No generative AI tools were used in this work nor preparation.

\end{document}